\documentclass[conference]{IEEEtran}

\IEEEoverridecommandlockouts
\usepackage{amsmath,amsthm,amsfonts}
\usepackage{graphicx}
\usepackage{epsf}
\usepackage{mdwmath}
\usepackage{mdwtab}
\usepackage{cite}
\usepackage{amssymb}
\usepackage{algorithmic}
\usepackage{algorithm}

\usepackage{amssymb}
\usepackage{algorithmic}
\usepackage{algorithm}

\newtheorem{thm}{Theorem}[section]

\newtheorem{cor}[thm]{Corollary}
\newtheorem{definition}[thm]{Definition}

\floatname{algorithm}{Sequential Thresholding \hspace{50cm}}

\newcommand{\E}{{\cal{E}}}

\newcommand{\Sc}{\mathcal{S}}
\def\R{\mathbb{R}}
\def\E{\mathbb{E}}
\def\P{\mathbb{P}}
\def\1{\mathbf{1}}
\def \btheta{\mbox{\boldmath$\theta$}}
\def\S{\mathcal{S}}
\def \hyp{\underset{\theta_0}{\overset{\theta_1}{\gtrless}}}

\begin{document}
\pagestyle{empty}

\title{Sequential Analysis in High-Dimensional \\ Multiple Testing and Sparse Recovery}

\author{\IEEEauthorblockN{Matthew Malloy}
\IEEEauthorblockA{ Electrical and Computer Engineering\\
University of Wisconsin-Madison\\
Email: mmalloy@wisc.edu}
\and
\IEEEauthorblockN{Robert Nowak}
\IEEEauthorblockA{
Electrical and Computer Engineering\\
University of Wisconsin-Madison\\
Email: nowak@ece.wisc.edu}}
\maketitle
\vspace{-2cm}

\begin{abstract}
This paper studies the problem of high-dimensional multiple testing and sparse recovery from the perspective of sequential analysis. In this setting, the probability of error is a function of the dimension of the problem. A simple sequential testing procedure is proposed. We derive necessary conditions for reliable recovery in the non-sequential setting and contrast them with sufficient conditions for reliable recovery using the proposed sequential testing procedure.  Applications of the main results to several commonly encountered models show that sequential testing can be exponentially  more sensitive to the difference between the null and alternative distributions (in terms of the dependence on dimension), implying that subtle cases can be much more reliably determined using sequential methods.
\end{abstract}

\section{Introduction}
High-dimensional testing and sparse recovery problems arise in a broad range of scientific and engineering applications.  The basic problem is summarized as follows.  Let $\btheta \in \R^n$ denote a parameter vector.  The dimension $n$ may be very large (thousands or millions or more), but $\btheta$ is sparse in the sense that most of its components are equal to a baseline/null value denoted by $\theta_0$ (e.g., $\theta_0 = 0$). The support of the sparse subset of components that deviate from the baseline is denoted by $\S$:
\begin{eqnarray} \label{eqn:undstats1}
  \S = \{i: \theta_i \neq \theta_0\}.
\end{eqnarray}
The parameter $\btheta$ is observed stochastically according to
\begin{eqnarray} \label{eqn:underlyingstats}
 y_i \sim f(y_i|\theta_i)
\end{eqnarray}
where $f(\cdot|\theta)$ is a parametric family of densities indexed by a scalar parameter $\theta \in \R$. The goal of the high-dimensional testing and sparse recovery problem is to identify $\S$ from observations of this form.

The conventional theoretical treatment of this problem assumes that a set of observations
are collected prior to data analysis.  Typically, in what we refer to as the non-sequential setting, each of the $n$ components is measured (one or more times) according to the model above {\em and then} component-wise tests are performed to estimate $\S$.

This papers investigates the high-dimensional testing problem from the perspective of sequential analysis.  In this setting, observations are gathered sequentially and adaptively, based on information gleaned from previous observations.  This allows the observation process to focus sensing resources on certain components at the expense of ignoring others.  For example, the process might first measure each component once, then focus on a reduced subset of `interesting' components in a second pass.  Such approaches have attracted attention lately due to importance in the biological sciences. They are also relevant in communications problems including spectrum sensing in cognitive radio, one of the motivations for our work.

To compare sequential and non-sequential methods we impose a budget on the total number of observations that can be made.  The main results show sequential methods can be dramatically more sensitive to small differences between the baseline/null $\theta_0$ and the alternative values of $\theta_i$.   Our approach is similar to the so-called {\em distilled sensing} method proposed in \cite{haupt} \cite{5513489}, however there are two main distinctions.  First, the results in this paper are applicable to a large class of problems characterized by one-sided tests; the distilled sensing approach is specific to the Gaussian setting. Second, here we are concerned with the probability of error in identifying $\S$; distilled sensing controls the false discovery and non-discovery rates which is less demanding than the error rate control.  

To give a sense of the main results, consider the case in which  $f(\cdot|\theta)$ is a Gaussian with mean $\theta$ and variance $1$.  If $\theta_0 = 0$ and the alternative is $\theta_i = \theta_1  > 0$ for $i \in \S$, then reliable detection (probability of error tending to zero as $n\rightarrow \infty$) is \emph{not} possible using non-sequential methods if $\theta_1 < \sqrt{2\log n}$. In contrast, a sequential method that we will demonstrate is reliable as long as $\theta_1 > \sqrt{4 \log|\S|}$, where $|\S|$ is the cardinality of the support set.  This shows the sequential method is more sensitive whenever $|\S| < n^{1/2}$; i.e., the sparse setting.  The improvement is especially remarkable when $\btheta$ is very sparse; e.g., if $|\S| \approx \log n$, then sequential methods succeed as long as $\theta_1$ is larger than a constant multiple of $\sqrt{\log\log n}$. The gains provided by the sequential method are even more pronounced for certain one-sided distributions.   In spectrum sensing (where the measurements follow gamma distributions), to within constant factors, if the SNR grows as $\log (|\S| \log n)$ then the sequential method is reliable, but any non-sequential procedure is unreliable if the SNR grows slower than $\sqrt{n}$.  To dramatize this result, if $|\S| \approx \log n$, then the gap between these conditions is {\em doubly} exponential in $n$.

\section{Problem Statement} \label{sec:ProbSetup}
For $i = 1,...,n$ let $y_i$ be a random variable distributed according to (\ref{eqn:underlyingstats}).  We say component $i$ follows the null distribution if $i \not \in \S$, where $\S$ is defined in (\ref{eqn:undstats1}), thus $y_i \sim f(\cdot|\theta_0)$.  Conversely, we say that component $i$ follows the alternative if $i \in \S$, and $\theta_i \neq \theta_0$.  Define $s := |\S|$ as the level of sparsity, and assume that $s \ll n$.  Our goal is exact recovery of $\S$.

With some abuse of notation, let $\theta_1$ be a parameter corresponding to an alternative distribution (\emph{not} the parameter corresponding to component $1$ of the vector $\btheta$). For simplicity we let $\theta_i = \theta_1$ for $i \in \S$, that is, all components following the alternative follow the same distribution $f(\cdot|\theta_1)$. This allows a simple binary hypotheses to test for inclusion in $\S$.  More general consideration could test composite alternatives (in which the alternative distributions are a family of distributions with unknown $\theta_i \neq \theta_0$) -- in this setting, the results of this paper can be viewed as quantifying the \emph{minimum} separation required between the null and any of the alternative distributions.  Extensions of this sort  are obvious in the context of the applications considered in Section \ref{sec:Apps}.

To index multiple independent identically distributed observations of $y_i$, we introduce a second subscript $j$ -- $y_{i,j}$ is the $j$th observation of the $i$th component.  Define the log-likelihood ratio statistic corresponding to index $i$ as
$$T_{i,m} := \sum_{j=1}^m \log \frac{ f(y_{i,j}|\theta_1)}{f(y_{i,j}|\theta_0)} \ .  $$
The distribution of the log-likelihood ratio depends on the number of independent observations, indicated by the subscript $m$.  As both sequential and non-sequential tests compare $T_{i,m}$ to a threshold, we refer to $T_{i,m}$ as the \emph{test statistic}.

\subsection{Measurement Budget}
To compare sequential and non-sequential methods, we impose a budget on the total number of measurements.  A single measurement consists of observing, for example, $y_{1,1}$ and thus, observing $(y_{1,1},...,y_{1,n})$ requires $n$ measurements.  The total number of measurements is limited to $N \leq 2mn$, where $m\geq 1$ is an integer.

\subsection{Non-Sequential Testing}
The non-sequential approach distributes the measurement budget uniformly over the $n$ components, making $2m$ i.i.d.\ observations of each.  Let $y_{i,1},\dots,y_{i,2m}$ denote the $2m$ observations of component $i$, and let $T_{i,2m}$ denote the corresponding test statistic.  The test takes the form
\begin{eqnarray}
\label{non-seq}
T_{i,2m}& \hyp & \tau \
\end{eqnarray}
and, for some $\tau$, is optimal in terms of probability of error among all (non-sequential) component-wise estimators. The estimated support set at threshold $\tau$ is
$$\S_\tau := \{i \ : \ T_{i,2m}>\tau\} \ . $$

\subsection{Sequential Thresholding} \label{sec:SDS}

The sequential method we propose is based on the following simple bisection idea. Instead of aiming to identify the components in $\S$, at each step of the sequential procedure we aim to eliminate about $1/2$ of the remaining components \emph{not} in $\S$ from further consideration.  The components that remain under consideration after $K$ such steps is our estimate of the set $\S$.

Suppose we begin by using half of our measurement budget to collect $m$ observations of each component.  The test statistic for each is $T_{i,m}$, a function of $(y_{i,1},\dots,y_{i,m})$.  Assume $\theta_0$ is known and let $T_{i,m}|\theta_0$ denote the random variable whose distribution is that of the test statistic under the null.
Consider the threshold test
$$T_{i,m} \ > \ \mathrm{median}(T_{i,m}|\theta_0) \ . $$
For $i \not \in \S$, the test statistic $T_{i,m}$ falls below $\mathrm{median}(T_{i,m}|\theta_0)$ with probability $1/2$.
The threshold test above thus eliminates approximately $1/2$ of the components that follow the null.  We can next use a portion of our remaining budget of $mn$ to repeat the same measurement and thresholding procedure on the remaining components.  Since approximately $n/2$ components remain this will require $m n/2$ of the remaining budget.  Repeating this process for sufficiently many iterations will remove, with high probability, all of the null components. We call this process \emph{sequential thresholding} and give a formal algorithm below.  The output of the procedure, $\S_K$, is the estimated support set.  Notice that sequential thresholding does not require prior knowledge of the size of the support set.

\begin{algorithm}[h]
\caption{}
\begin{algorithmic} \label{alg:sds}
\STATE{input: $K>0$ steps, $\gamma_0 := \mathrm{median}(T_{i,m}|\theta_0)$}
\STATE{initialize: $\S_0 = \{1,...,n\}$}
\FOR{$k = 1,\dots,K$}
\FOR{$ i \in \S_{k-1}$}
\STATE{{\bf measure}: $\{y^{(k)}_{i,j}\}_{j=1}^{m} \sim \left\{
           \begin{array}{ll}
             \prod_{j =1}^{m} f(y_{i,j}^{(k)}|\theta_0) & i \not \in \mathcal{S} \\
             \prod_{j =1}^{m} f(y_{i,j}^{(k)}|\theta_1) & i \in \mathcal{S}
           \end{array}
         \right.$ }
\STATE{{\bf threshold}: $\S_k \ := \ \{i\in\S_{k-1} \ : T_{i,m}^{(k)} > \gamma_0 \}$}
\ENDFOR
\ENDFOR
\STATE{output: $\S_K$}
\end{algorithmic}
\end{algorithm}

While sequential thresholding is described and analyzed using a threshold at the median of the null, in practice, other thresholds can be used (for example, a threshold at the $95$ percentile) and can result in improved performance.

\subsection{Sequential Thresholding Satisfies Budget}
The number of measurements used by sequential thresholding satisfies the overall measurement budget $N\leq 2mn$ in expectation.  The expected number of measurements is
\begin{eqnarray*}
  \E\left[\sum_{k=0}^{K-1}|\S_k| \right]  &\leq & \sum_{k=0}^{K-1} \left(\frac{m(n-s)}{2^k} + ms\right) \\
  & \leq & 2m(n-s) + msK \
\end{eqnarray*}
since, in expectation, we eliminate half of the remaining null components (and perhaps some following the alternative, hence the first inequality) on each of the $K$ passes.  Our interest is in high-dimensional limits of $n$ and $s$.  Suppose that $sK$ grows sublinearly with $n$.  Then for any $\epsilon >0$ there exists an $N_\epsilon$ such that  $\E\left[\sum_{k=0}^{K-1}|\S_k| \right]  \leq 2(1+\epsilon)mn$ for every $n>N_\epsilon$.  We suppress the factor $1+\epsilon$ as we proceed as it does not effect our results.

\subsection{Implementations}

There are two possible implementations of sequential thresholding which we refer to as {\em parallel} and {\em scanning}. \\ \vspace{-.12in} \\
\noindent {\bf parallel}: The parallel implementation measures and tests all $n$ components in parallel according to the procedure. \\  \vspace{-.12in} \\
\noindent {\bf scanning}:  The scanning implementation measures and tests the $n$ components in a sequence (which can be arbitrary).  For example, the scanning implementation can begin with component $i=1$ and repeatedly measure and threshold the observations up to $K$ times.  If an observation falls below the threshold at any point, then the scanning procedure immediately moves on to the next component.  If $K$ observations are made without an observation falling below the threshold, then the component is added to the set $S_K$.  The expected number of observations obeys the same bound as derived above. \\  \vspace{-.12in} \\
\noindent The two implementations are equivalent from a theoretical perspective.  The parallel implementation may be more natural for large-scale experimental designs (e.g., in the biological sciences), whereas the scanning implementation is more appropriate in communications applications such as spectrum sensing.  The latter also reveals natural connections between sequential thresholding and sequential probability ratio tests.

\subsection{Connection to Sequential Probability Ratio Tests}

As we will show in the following section, in the high-dimensional limit ($n\rightarrow \infty$) sequential thresholding can drive the probability of error to zero if the divergence between the null and alternative distributions is $\log |\S|$ times a small constant.  This specializes in the Gaussian setting to the requirement that the difference between the means is at least $\sqrt{4 \log |\S|}$, which compares favorably to the requirement that the difference exceeds $\sqrt{2\log n}$ for non-sequential methods.

In fact, the $\log |\S|$ dependence of sequential thresholding is optimal, up to constant factors.
This follows from well-known results in sequential testing.  Let $\widehat \S$ denote the result of any testing procedure based on $n$ local (component-wise) tests of the form $H_0$: $i \not \in \S$ against $H_1:$ $i\in \S$.  Each test is based on the sequential observations $y_{i,1},y_{i,2},\dots,y_{i,N}$, and the {\em stopping time} of the test is the value of $N$ (possibly random) when the decision is made.

Suppose that each individual test has false-positive and false-negative error probabilities less than $\alpha:=\epsilon/(n-|\S|)$ and $\beta:=\epsilon/|\S|$, respectively.  Then the expected total number of errors is $\E|\widehat S \cap \S^C|+\E|\widehat S^{C}\cap \S| \leq 2\epsilon$.  It is necessary that this expected number tend to zero in order for the probability of error, $\P(\widehat S\neq \S)$, to tend zero.  With the above specifications for the two types of error, it is possible to design a sequential probability ratio test (SPRT) for each component.

The SPRT computes a sequence of likelihood ratios, where $\ell_{i,n}$ is the likelihood
ratio of $y_{i,1},\dots,y_{i,n}$, $n\geq 1$.  The SPRT terminates when $\ell_{i,n}\geq B$ or $\ell_{i,n}\leq A$, where the thresholds $A$ and $B$ are determined by the equations
$\alpha = B^{-1}(1-\beta)$ and $\beta = A(1-\alpha)$ (see \cite{SeqAnalysis} p. 11).  Note that, unlike sequential thresholding, the SPRT requires knowlege of both distributions as well as the level of sparsity.
Since such information is usually unavailable in applications, we advocate the use of sequential thresholding instead; it requires only crude knowledge of the null and nothing about the
alternative or sparsity level.

From the Wald equation, the expected stopping time of the SPRT per index is (approximately) \cite{SeqAnalysis}
\begin{eqnarray*}
\E_0 [N'] & \cong & \mu_0^{-1} \left\{ \alpha \log\left(\frac{1-\beta}{\alpha}\right) \, + \, (1-\alpha) \log\left(\frac{\beta}{1-\alpha}\right)\right\} \\
\E_1 [N'] & \cong & \mu_1^{-1} \left\{ (1-\beta) \log\left(\frac{1-\beta}{\alpha}\right) \, + \, \beta \log\left(\frac{\beta}{1-\alpha}\right) \right\}  \
\end{eqnarray*}
where $\E_i$ denotes the expectation under $f(\cdot|\theta_i)$ and $\mu_i := \E_i\left[\log\frac{f(y|\theta_1)}{f(y|\theta_0)}\right]$, $i=0,1$.
In our case $\alpha = \epsilon/(n-|\S|)$ and $\beta = \epsilon/|\S|$, and as $\epsilon \rightarrow 0$
we have
\begin{eqnarray*}
\E_0 [N'] & \cong & \mu_0^{-1} \log \frac{\epsilon}{|\S|} \\
\E_1 [N'] & \cong & \mu_1^{-1} \log \frac{n-|\S|}{\epsilon} \ .
\end{eqnarray*}
If $|\S|\ll n$, then the expected total number of measurements of made by all $n$ SPRTs is
\begin{eqnarray*}
E[N] = (n-|\S|) \E_0 [N'] + |\S| \E_1 [N'] & \cong & \frac{n}{\mu_0} \log\frac{\epsilon}{|\S|} \ .
\end{eqnarray*}
Note that $\mu_0 = -D_0:= -D(f(\cdot|\theta_0)||f(\cdot|\theta_1))$, the KL divergence of $f(\cdot|\theta_1)$ from $f(\cdot|\theta_0)$, so expected total number of observations made by the $n$ SPRTs is
$$E[N] \cong \frac{n}{D_0} \log\frac{|\S|}{\epsilon} \ . $$
It follows from the optimality of the SPRT that no other component-wise testing procedure with $\epsilon$ error-rate requires fewer observations.
Now let us constrain this expected total to be less than or equal to $2mn$.  This yields a necessary condition for controlling the probability of error of any sequential test:
$$D(f(\cdot|\theta_0)||f(\cdot|\theta_1)) \ \gtrsim \ \frac{1}{2m}\log \frac{|\S|}{\epsilon} \ . $$

\section{Main Results}
The main results rely on the extremal properties of the test statistic.  We say that a testing procedure is \emph{reliable} if it drives the probability of error to zero in the high-dimensional limit.  More formally, consider a sequence of multiple testing problems indexed by dimension $n$.  Let $\S(n)$ denote the true support set and let $\widehat{\S}(n) = \S_\tau$ (non-sequential procedure with threshold $\tau$) or $\widehat{\S}(n) = \S_K$ (sequential procedure with $K$ passes).  We define a notion of reliability as follows.

\begin{definition}{\bf (Reliability)} Let ${\cal E}$ denote the error event $\{\widehat{\S}(n) \neq \S(n)\}$. We say that the support set estimator $\widehat{\S}(n)$ is {\em reliable} if
$\lim_{n\rightarrow\infty} \P({\cal E}) = 0$; conversely, an estimator $\widehat{\S}(n)$ is \emph{unreliable} if $\lim_{n\rightarrow\infty} \P({\cal E}) > 0$.
\end{definition}

To simplify notation we will not explicitly indicate the dependence of the statistics on $n$. We show that the non-sequential testing procedure in (\ref{non-seq})  is {\em unreliable} at every threshold level $\tau$ if
\begin{eqnarray} \label{eqn:an}
 \lim_{n\rightarrow \infty}\mathbb{P}\left(\frac{\max_{i\not \in \S}  T_{i,2m}}{\mathrm{median}( T_{i,2m}|\theta_1)} \geq 1\right) = 1  \ .
\end{eqnarray}
Conversely, sequential testing according to sequential thresholding is \emph{reliable} if
\begin{eqnarray}  \label{eqn:bn}
\lim_{n\rightarrow \infty}\mathbb{P}\left(\frac{\min_{k=1}^K \min_{i \in \S}T_{i,m}^{(k)} }{\mathrm{median}\left( T_{i,m} |\theta_0 \right)} \leq 1\right) = 0 \ ,
\end{eqnarray}
and $K = (1+\epsilon)\log_2 n$, for any $\epsilon >0$.  We are interested in ranges of $\theta_1 >\theta_0$ that satisfy the conditions above.  In many cases of interest, (\ref{eqn:an}) and (\ref{eqn:bn}) hold simultaneously for a wide range parameter values. This  implies that there are many regimes in which sequential methods are reliable, but non-sequential methods are not.

For example, we show in Section~\ref{sec:AWGN} that if the underlying component distributions are unit variance Gaussian with means $\theta_0=0$ and $\theta_1>0$, then the non-sequential procedure (\ref{non-seq}) is unreliable if $\theta_1 < \sqrt{\frac{1}{m}\log n}$ whereas sequential thresholding is reliable if $\theta_1 \geq \sqrt{\frac{2}{m}\log(s \log_2 n)}$.  The size of the sparse support, $s$, is typically much smaller than the overall dimension $n$, and so there are many cases in which the sequential method is \emph{reliable} but the non-sequential method is \emph{unreliable}. The gap between the two conditions can be exponentially large in terms of the dimension $n$. As a specific example, if $s = \log n$, then the sequential method is reliable if $\theta_1 \geq 2 \sqrt{\log\log n}$ and the non-sequential method is unreliable if $\theta_1 \leq \sqrt{2 \log n}$.




\subsection{Limitation of Non-Sequential Testing}
\begin{thm} \label{lemma:NAanbn}
If (\ref{eqn:an}) holds, then the non-sequential procedure in (\ref{non-seq}) is unreliable.  Specifically, if ${\cal E}_\tau$ is the error event $\{\S_\tau \neq \S\}$, then for every $\tau$
\begin{eqnarray*}
\lim_{n\rightarrow \infty}\P\left( {\cal E}_\tau\right) & \geq & \frac{1}{2}.
\end{eqnarray*}
\end{thm}
\begin{proof}
The non-sequential testing procedure accepts the null hypothesis if the test statistic $T_{i,2m}$ is less than some threshold, $\tau$, and conversely, rejects the null hypothesis if $T_{i,2m} \geq \tau$.  The probability of error at threshold level $\tau$ is
\begin{eqnarray*} \nonumber
 \P\left( {\cal E}_\tau\right) & =  &  \mathbb{P}\left(\bigcup_{i \not \in \mathcal{S}}\{T_{i,2m} \geq \tau\} \bigcup_{i  \in \mathcal{S}} \{T_{i,2m} < \tau\} \right) \, ,   \label{eqn:PEhalf}
\end{eqnarray*}
and the {\em minimum} probability of error is $\min_\tau \P\left( {\cal E}_\tau\right) $.
Now suppose we take $\tau = \mbox{median}(T_{i,2m}|\theta_1)$, the median value of the test statistic under the alternative.  At this threshold level, the false-negative rate would be $1/2$, and so the overall probability of error would be at least $1/2$. It follows that the minimum probability of error can be bounded from below by
\begin{eqnarray*} \nonumber
  \min_\tau \P\left( {\cal E}_\tau\right)  \geq   \min\left(1/2 \, , \, \mathbb{P}(\cup_{i \not \in \mathcal{S}}\{T_{i,2m} \geq \mbox{median}(T_{i,2m}|\theta_1 )\}\right)  .
\end{eqnarray*}
According to (\ref{eqn:an}) the second argument above tends to $1$ as
$n\rightarrow \infty$, which completes the proof.
\end{proof}

\subsection{Capability of Sequential Thresholding}
\begin{thm} \label{lemma:adapt}
If (\ref{eqn:bn}) holds, then sequential thresholding is reliable if $K = (1+\epsilon)\log_2 n$, for $\epsilon >0$.  Specifically, if ${\cal E}_\epsilon$ is the error event $\{\S_K \neq \S\}$, then for any $\epsilon >0$
\begin{eqnarray*}
\lim_{n\rightarrow \infty} \P\left({\cal E}_\epsilon \right) & = &0.
\end{eqnarray*}
\end{thm}
\begin{proof}
The probability of error is
\begin{eqnarray}
    \P({\cal E}_\epsilon) &:=&\mathbb{P} \left( {\Sc}_{K} \neq \Sc \right) \nonumber \\
    &=& \mathbb{P}\left(\left\{\Sc \cap {\Sc}_{K}^c \neq \emptyset  \right\} \cup \left\{\Sc^c \cap {\Sc}_{K} \neq  \emptyset  \right\} \right) \label{eqn:PeDecomp} \nonumber \\
    &\leq& \mathbb{P}\left(\Sc \cap {\Sc}_{K}^c \neq \emptyset  \right) +\mathbb{P} \left(\Sc^c \cap {\Sc}_{K} \neq  \emptyset  \right) \ ,
  \end{eqnarray}
  where the superscript $c$ denotes the complementation of the set.  The upper bound on the probability of error consists of two terms, the false-negative and false-positive probabilities.  The false positive probability (second term in (\ref{eqn:PeDecomp})) can be bounded as follows:
  \begin{eqnarray*}
 &&\hspace{-1cm} \mathbb{P}\left( \Sc^c \cap {\Sc}_{K} \neq  \emptyset  \right)  \\
 & =& \mathbb{P} \left( \bigcup_{i \not \in \mathcal{\Sc}} \bigcap_{k=1}^{K} \left\{T_{i,m}^{(k)} \geq \mathrm{median}\left( T_{i,m}| \theta_0\right)\right\} \right) \\
 &\leq& \sum_{i\not \in \mathcal{\Sc}} \left(\mathbb{P} \left( T_{i,m}^{(1)} \geq \mathrm{median}\left( T_{i,m} |\theta_0\right) \right)\right)^K \\
 &=&   \frac{n-|\S|}{2^K} \label{eqn:p2sim}
\end{eqnarray*}
where the last step follows since the probability a random variable exceeds its median is $1/2$.  Since $K = (1+\epsilon) \log_2 n$, with $\epsilon>0$, we have
\begin{eqnarray*} \label{eqn:falsePosContro}
  \lim_{n \rightarrow \infty}  \mathbb{P}\left( \S^c \cap \S_K \neq  \emptyset  \right) = 0 \ .
\end{eqnarray*}

Bounding the false-negative probability (first term  in (\ref{eqn:PeDecomp})) depends on the distribution of the test statistic under the alternative $\theta_1$:
\begin{eqnarray*} \label{eqn:P1sim} \nonumber
&&\hspace{-1cm}\mathbb{P}\left(S \cap {S}_{K}^c \neq \emptyset  \right)   \\
&=& \mathbb{P}\left( \bigcup_{k=1}^K \bigcup_{i\in \mathcal{S}} \left\{T_{i,m}^{(k)} \leq \mathrm{median} \left(T_{i,m}| \theta_0 \right)\right\} \right) \\ \label{eqn:adpmed}
&=& \mathbb{P}\left( \min_{k=1}^{K} \min_{i\in \S} T_{i,m}^{(k)} \leq \mathrm{median} \left(T_{i,m}| \theta_0 \right) \right)
\end{eqnarray*}
which, from (\ref{eqn:bn}), goes to zero in the limit, completing the proof.
\end{proof}

\section{Applications}
\label{sec:Apps}
To illustrate the main results we consider three canonical settings arising in high-dimensional multiple testing.  We again have in mind a sequence of problems
and consider behavior in the high-dimensional limit.  Thus, when we write $\theta \leq g(n)$ (or $\theta\geq g(n)$) we mean that the parameter $\theta$ may (must) grow with dimension $n$ no faster (slower) than the function $g(n)$.

\subsection{Gaussian Model} \label{sec:AWGN}
Gaussian noise models are commonly assumed in multiple testing problems arising in the biological sciences (e.g., testing which of many genes or proteins are involved in a certain process or function).  For example, a multistage testing procedure similar in spirit to sequential thresholding was used to determine genes important for virus replication  in \cite{hao}. Consider a high-dimensional hypothesis test in additive Gaussian noise where the parameter $\btheta$ represents the mean of the distribution. We assume the null hypothesis follows zero mean ($\theta_0 = 0$), unit variance gaussian statistics; the alternative hypothesis, mean $\theta_1 > 0$, unit variance:
\begin{eqnarray*} \label{eqn:Gaussstats}
    y_i  \sim \left\{
           \begin{array}{ll}
             {\cal{N}}(0,1) \, ,  &  i \not \in \mathcal{S} \\
             {\cal{N}}(\theta_1,1) \, , & i \in \mathcal{S}.
           \end{array}
         \right.
\end{eqnarray*}

\subsubsection{{Non-Sequential Testing}}
We make $2m$ measurements of each component of $\btheta$.  The test statistic again follows a normal distribution:
\begin{eqnarray} \label{eqn:GausteststatsNA}
    T_{i,2m} = \frac{1}{2m}\sum_{j=1}^{2m} y_{i,j}    \sim \left\{
           \begin{array}{ll}
             {\cal{N}}\left(0,\frac{1}{2m}\right)  \, ,  & i \not \in \mathcal{S} \\
             {\cal{N}}\left(\theta_1,\frac{1}{2m}\right)  \, ,  & i \in \mathcal{S}.
           \end{array}
         \right.
\end{eqnarray}
\begin{cor} \label{cor:gausNS}
If $\theta_1 < \sqrt{\frac{\log (n-s)}{m} }$, then the non-sequential testing procedure in (\ref{non-seq}) is unreliable, i.e., $\min_{\tau} \mathbb{P}({\cal{E_{\tau}}}) \geq 1/2$.
\end{cor}
\begin{proof}
For the test statistic in equation (\ref{eqn:GausteststatsNA}), we satisfy (\ref{eqn:an}) provided $\mathrm{median}\left(T_{i,2m}| \theta_1 \right) \leq \sqrt{\frac{\log (n-s)}{m} }$ (see, for example \cite{lea:lin:roo83}).  By Theorem~\ref{lemma:NAanbn} and since  $\mathrm{median}\left(T_{i,2m}| \theta_1 \right) = \theta_1$, if
\begin{eqnarray*}
  \theta_1 \leq \sqrt{\frac{ \log (n-s) }{m} }
\end{eqnarray*}
then non-sequential thresholding is unreliable.
\end{proof}

\subsubsection{{Sequential Testing}}
Sequential thresholding makes $m$ measurements of each component in the set $\S_k$ at each step.  The test statistic follows a normal distribution:
\begin{eqnarray} \label{eqn:GausteststatsAd}
    T_{i,m}^{(k)} = \frac{1}{m}\sum_{j=1}^{m} y_{i,j}   \sim \left\{
           \begin{array}{ll}
             {\cal{N}}\left(0,\frac{1}{m}\right)  & i \not \in \mathcal{S} \\
             {\cal{N}}\left(\theta_1,\frac{1}{m}\right)  & i \in \mathcal{S}.
           \end{array}
         \right.
\end{eqnarray}
\begin{cor} \label{cor:gausS}
If $\theta_1 > \sqrt{\frac{2}{m}\log (s \log_2 n)}$, then
sequential thresholding is reliable.
\end{cor}
\begin{proof}In this case, equation (\ref{eqn:bn}) is satisfied provided $\mbox{median}(T_{i,m}|\theta_0) \leq {\theta_1} - \sqrt{\frac{2\log Ks}{m}}$ (see for example \cite{lea:lin:roo83}).   Since $\mbox{median}(T_{i,m}|\theta_0) = 0$, Theorem \ref{lemma:adapt} tells us that provided
\begin{eqnarray}
 \theta_1 \geq \sqrt{\frac{2}{m}{\log Ks}}
\end{eqnarray}
with $K = (1+ \epsilon) \log_2 n$, we reliably recover $\mathcal{S}$.

\end{proof}

\subsection{Gamma Model: Spectrum Sensing} \label{sec:SS}
Often termed \emph{hole detection}, the objective of spectrum sensing is to identify unoccupied communication bands in the electromagnetic spectrum.  Most of the bands will be occupied by {\em primary} users, but these users may come and go, leaving certain bands momentarily open and available for {\em secondary} users.  Recent work in spectrum sensing has given considerable attention to such scenarios, including some work employing adaptive sensing methods (see, for example \cite{Castro}, \cite{5454086}).

Following the notation throughout this paper, channel occupation is parameterized by $\btheta$, with $\theta_0$ denoting the \emph{signal plus noise} power in the occupied bands, and $\theta_1$ representing the \emph{noise only} power in the un-occupied bands.  Without loss of generality, we let $\theta_1 = 1$.  The statistics follow a complex Gaussian distribution -- $y_i \sim {\cal{CN}}\left(0,\theta\right)$.  From \cite{1447503}, making $m$ measurements of each index, the likelihood ratio test statistic follows a $\mathrm{Gamma}$ distribution:
\begin{eqnarray} \label{eqn:GammateststatsNA}
T_{i,m}^{(k)} = \sum_{j=1}^m |y_{i,j}|^2 \sim
\left\{
  \begin{array}{ll}
    \mathrm{Gamma}\left({m},\theta_0\right) & i \not \in \mathcal{S} \\
    \mathrm{Gamma}\left({m},1\right) & i \in \mathcal{S}.
  \end{array}
\right.
\end{eqnarray}
Remarkably, the sequential testing procedure is reliable, to within constant factors, if $\theta_0$ grows as $\log (s \log_2n)$, but the non-sequential testing procedure is unreliable if $\theta_0$ grows as $(n-s)^{\frac{1}{2m}}$.  This implies, if $s = \log n$, then the gap between these conditions is {\em doubly} exponential in $n$.

Since we are interested in detecting the sparse set of vacancies in the spectrum, our hypothesis test is reversed.  We reject the null hypothesis (occupied component) if the
test statistic falls {\em below} (rather than above) a certain threshold.  In this case, the inequalities in the key conditions (\ref{eqn:an}) and (\ref{eqn:bn}) are reversed: specifically, the non-sequential thresholding procedure is \emph{unreliable} if
\begin{eqnarray} \label{eqn:an2}
 \lim_{n\rightarrow \infty}\mathbb{P}\left(\frac{\min_{i\not \in \S}  T_{i,2m}}{\mathrm{median}( T_{i,2m}|\theta_1)} \leq 1\right) = 1  \
\end{eqnarray}
and sequential thresholding is \emph{reliable} if
\begin{eqnarray}  \label{eqn:bn2}
\lim_{n\rightarrow \infty}\mathbb{P}\left(\frac{\max_{k=1}^K \max_{i \in \S}T_{i,m}^{(k)} }{\mathrm{median}\left( T_{i,m} |\theta_0 \right)} \geq 1\right) = 0 \ .
\end{eqnarray}

\subsubsection{{Non-Sequential Testing}}
In the non-sequential procedure (\ref{non-seq}), we make $2m$ measurements per index.  The distribution of the test statistic follows a gamma distribution with shape parameter $2m$.
\begin{cor}
If $\theta_0 < 2(m-1) (n-s)^{\frac{1}{2m}}$, then the non-sequential procedure in (\ref{non-seq}) is unreliable.
\end{cor}
\begin{proof}
In this case, because the hypothesis test is reversed, we aim to satisfy (\ref{eqn:an2}).
Since $\mbox{median}(T_{i,2m} | \theta_1) \geq 2(m-1)$, we have
\begin{eqnarray*} \nonumber
\mathbb{P}\left( \frac{\min_{i \not \in \Sc} T_{i,2m} }{\mathrm{median}\left(T_{i,2m} | \theta_1 \right)  } \leq 1 \right) \geq \mathbb{P}\left( \frac{\min_{i \not \in \Sc} T_{i,2m}}{2(m-1)}  \leq 1 \right).
\end{eqnarray*}
If $ 2(m-1) > \frac{\theta_0}{(n-s)^{\frac{1}{2m}}}$, the right hand side above goes to $1$ as $n$ grows large (see Appendix \ref{sec:AppC}).
Together with Theorem \ref{lemma:NAanbn} this implies that if $ \theta_0 < 2(m-1) (n-s)^{\frac{1}{2m}} $ then the non-sequential procedure is unreliable.
\end{proof}

\subsubsection{{Sequential Testing}}
Sequential thresholding makes $m$ measurements of each component in the set $\S_k$ at each step.  The test statistic follows the Gamma distributions in (\ref{eqn:GammateststatsNA}).

\begin{cor}
If $\theta_0 > \frac{\log (s \log_2n)}{m}$, then sequential thresholding is reliable.
\end{cor}
\begin{proof}
It suffices to show (\ref{eqn:bn2}) is satisfied. For all $m$ and $\theta_0$, we have $\mathrm{median}(T_{i,m}|\theta_0) \geq \theta_0(m-1)$.  We upper bound (\ref{eqn:bn2}) by
\begin{eqnarray} \nonumber
 \lim_{n\rightarrow \infty} \mathbb{P}\left(  \frac{\max_{k=1}^{K} \max_{i \in \Sc} T_{i,m}^{(k)}} {\theta_0(m-1)} \geq 1  \right)
\end{eqnarray}
which goes to zero in the limit provided $\theta_0(m-1) > \log Ks$ (see Appendix \ref{sec:AppD}).  Together with Theorem \ref{lemma:adapt} if
\begin{eqnarray}
  \theta_0 > \frac{\log Ks}{m-1}
\end{eqnarray}
with $K = (1+ \epsilon)\log_2 n$, then sequential thresholding is reliable.
\end{proof}

\subsection{Poisson Model: Photon-based Detection} \label{sec:pois}
Lastly we consider a situation in which the component distributions are Poisson.  This model arises naturally in testing problems involving photon counting  (e.g.,  optical communications or biological applications using fluorescent markers).  We let the (sparse) alternative follow a Poisson with \emph{fixed} rate $\theta_1$, and the null hypothesis a rate $\theta_0$, $\theta_0 > \theta_1$:
\begin{eqnarray*}
 y_i \sim
\left\{
  \begin{array}{ll}
    \mathrm{Poisson}(\theta_0) & i \not \in \mathcal{S} \\
    \mathrm{Poisson}(\theta_1) & i \in \mathcal{S} \ ,
  \end{array}
\right.
\end{eqnarray*}
Note that as $\theta_0 > \theta_1$, our hypothesis test is reversed as in the spectrum sensing example (and equations (\ref{eqn:an2}) and (\ref{eqn:bn2})).

The test statistic is a sum of the individual measurements, again following a Poisson distribution.   In this setting, the gap between sequential and non-sequential testing is similar to that of the Gaussian case. Proofs are left to Appendices \ref{sec:AppE} and \ref{sec:AppF}.


\begin{cor}
For any fixed $\theta_1$, if $\theta_0 < \frac{\log(n-s)}{2m}$,
non-sequential thresholding is unreliable.
\end{cor}

\begin{cor}
For any fixed $\theta_1$, if  $\theta_0 > \frac{\log (s\log_2 n) +1}{m}$, sequential thresholding is reliable.
\end{cor}

\section{Conclusion}
This paper studied the problem of high-dimensional testing and sparse recovery
from the perspective of sequential analysis. The gap between the null parameter $\theta_0$ and the alternative $\theta_1$ plays a crucial role in this problem.  We derived necessary conditions for reliable recovery in the non-sequential setting and contrasted them with sufficient conditions for reliable recovery using the proposed sequential testing procedure.  Applications of the main results to several commonly encountered models show that sequential testing can be exponentially (in dimension $n$) more sensitive to the difference between the null and alternative distributions, implying that subtle cases can be much more reliably determined using sequential methods.

\bibliographystyle{IEEEtran}
\bibliography{AdapSensISIT}

\newpage

\appendix

\subsection{Gamma Non-Sequential} \label{sec:AppC}
The cumulative distribution function of $\mathrm{Gamma}(2m,\theta_0)$ is given as
\begin{eqnarray} \nonumber
F(\gamma) = 1- e^{-\frac{\gamma}{\theta_0}} \sum_{\ell=0}^{2m-1} \left(\frac{\gamma}{\theta_0} \right)^\ell \frac{1}{\ell!}
\end{eqnarray}
hence,
\begin{eqnarray*} \nonumber
  \mathbb{P}\left(\frac{\min_{i \not \in \S} T_{i,2m}}{\gamma} \leq 1 \right) =  1-\left(e^{-\frac{\gamma}{\theta_0}} \sum_{\ell=0}^{2m-1} \left(\frac{\gamma}{\theta_0} \right)^\ell\frac{1}{\ell!}\right)^{n-s}
\end{eqnarray*}
Letting $\gamma = \frac{\theta_0}{(n-s)^{\frac{1}{2m}}}$ and taking the limit, it can be shown
\begin{eqnarray} \nonumber
\lim_{n\rightarrow \infty} 1-\left(e^{-\left({(n-s)^{-\frac{1}{2m}}}\right)} \sum_{\ell=0}^{2m-1} \frac{(n-s)^{\frac{-\ell}{2m}}}{\ell!}\right)^{n-s}\\
 = 1-e^{-\frac{1}{(2m)!}}. \nonumber
\end{eqnarray}
If $\gamma > \frac{\theta_0}{(n-s)^{\frac{1}{2m}}}$, then
\begin{eqnarray} \nonumber
\mathbb{P}\left(\frac{\min_{i \not \in \S} T_{i,2m}}{\gamma} \leq 1 \right)  = 1.
\end{eqnarray}

\subsection{Gamma Sequential} \label{sec:AppD}
The cumulative distribution function of $\mathrm{Gamma}(m,1)$ is given as
\begin{eqnarray*}
F(\gamma) = 1- e^{-\gamma} \sum_{\ell=0}^{m-1} \frac{\gamma^\ell}{\ell!}
\end{eqnarray*}
hence,
\begin{eqnarray*}
  \mathbb{P}\left(\max_{k = 1}^{K} \max_{i \in \S} T_{i,m}^{(k)} \geq \gamma \right) = 1-\left(  1- e^{-\gamma} \sum_{\ell=0}^{m-1} \frac{\gamma^\ell}{\ell!} \right)^{Ks}.
\end{eqnarray*}
Letting $\gamma = (1+\epsilon) \log{Ks} $, for some $\epsilon > 0$, we have
\begin{eqnarray*}
  \lim_{n \rightarrow \infty} 1-\left(  1- \frac{1}{(Ks)^{1+\epsilon}} \sum_{\ell=0}^{m-1} \frac{\left((1+\epsilon) \log{Ks} \right)^\ell}{\ell!} \right)^{Ks}  = 0 \ .
\end{eqnarray*}

\subsection{Poisson Non-Sequential} \label{sec:AppE}
The likelihood ratio statistic is distributed as
\begin{eqnarray} \label{eqn:PoissonteststatsNA} \nonumber
T_{i,2m} = \sum_{j=1}^{2m} y_{i,j} \overset{iid}\sim
\left\{
  \begin{array}{ll}
    \mathrm{Poisson}(2m\theta_0) & i \not \in \mathcal{S} \\
    \mathrm{Poisson}(2m\theta_1) & i \in \mathcal{S}.
  \end{array}
\right.
\end{eqnarray}
It suffices to show
\begin{eqnarray*}
\lim_{n\rightarrow \infty} \mathbb{P}\left(  \frac{\min_{i \not \in \S} T_{i,2m}}{\mathrm{median}(T_{2m}| \theta_1)} \leq 1 \right) = 1.
\end{eqnarray*}
for any $\theta_0 < \frac{\log(n-s)}{2m}$.  The bound we derive is loose, but sufficient to show the adaptive scheme is superior.
First, we assume that $\mathrm{median}(T_{2m}| \theta_1) > 0 $.  Next we have
\begin{eqnarray*}
\mathbb{P}\left( \min_{i \not \in \S}  T_{i,2m} \leq {\mathrm{median}(T_{2m}| \theta_1) } \right) \\
& \ \hspace{-1.5in}  \geq \  \mathbb{P}\left( \min_{i \not \in \S}  T_{i,2m} =0 \right)\\
& \ \ \hspace{-1.3in}  =  \ 1-\left(1-e^{- 2m\theta_0 }\right)^{n-s} \ .
\end{eqnarray*}
If $ 2m \theta_0 < \log (n-s)$, then
\begin{eqnarray*}
\lim_{n\rightarrow \infty}  1-\left(1-e^{-2m \theta_0}\right)^{n-s}  = 1
\end{eqnarray*}
which is also true provided $\theta_0 < \frac{\log(n-s)}{2m}$ and concludes the proof.

\subsection{Poisson Sequential} \label{sec:AppF}
In sequential thresholding, for each $i \in S_k$
\begin{eqnarray} \label{eqn:PoissonteststatsAd} \nonumber
T_{i,m}^{(k)} = \sum_{j=1}^{m} y_{i,j} \overset{iid}\sim
\left\{
  \begin{array}{ll}
    \mathrm{Poisson}(m\theta_0) & i \not \in \mathcal{S} \\
    \mathrm{Poisson}(m\theta_1) & i \in \mathcal{S}.
  \end{array}
\right.
\end{eqnarray}

We need to show, for the test statistic above,
\begin{eqnarray*}
\lim_{n\rightarrow \infty} \mathbb{P}\left(  \frac{\max_{k=1}^K \max_{i \in \S} T_{i,m} }{\mathrm{median}(T_{m}| \theta_0 )} \geq 1 \right) = 0.
\end{eqnarray*}
First, we note $\mathrm{median}(T_{m}| \theta_0 ) \geq  m\theta_0-1$.  Hence,
\begin{eqnarray*}
\mathbb{P}\left(  \max_{k=1}^K \max_{i \in \S} T_{i,m} \geq \mathrm{median}(T_{m}| \theta_0 ) \right) \
\\  & \ \hspace{-2.8in} \leq & \hspace{-1.4in} \mathbb{P}\left(  \max_{k=1}^K \max_{i \in \S} T_{i,m} \geq m\theta_0-1 \right).
\end{eqnarray*}
We can bound the probability of a single event by Chernoff's bound \cite{Gubner:2006:PRP:1207291}, p.166.  For $T_{i,m} \sim \mathrm{Possion}({m\theta_1})$ we have:
\begin{eqnarray*}
  \mathbb{P}\left(T_{i,m} \geq \gamma \right) &\leq& e^{- {m\theta_1} - \gamma \left( \log\left(\frac{\gamma}{m\theta_1}\right)-1\right)} \\
  &\leq& e^{- \gamma \left( \log\left(\frac{\gamma}{m\theta_1}\right)-1\right)}.
\end{eqnarray*}
which implies
\begin{eqnarray*}
  \mathbb{P}\left( \max_{k=1}^K \max_{i \in \S} T_{i,m} \geq \gamma \right) \leq 1 - \left(1-e^{- \gamma \left( \log\left(\frac{\gamma}{m\theta_1}\right)-1\right)}\right)^{Ks}
\end{eqnarray*}
Letting $\gamma = \log Ks$ and taking the limit as $n \rightarrow \infty$ of the expression above for any fixed $\theta_1$, we conclude
\begin{eqnarray*}
\lim_{n \rightarrow \infty}  \mathbb{P}\left(\frac{ \max_{k=1}^K \max_{i \in \S} T_{i,m}}{\log Ks} \geq 1 \right) = 0
\end{eqnarray*}
Thus, if $\log Ks  \leq  m\theta_0 -1$, or equivalently
\begin{eqnarray*}
  \theta_0 \geq \frac{\log Ks +1}{m},
\end{eqnarray*}
sequential thresholding is reliable.

\end{document}